\documentclass{amsart}

\usepackage{amssymb,amsthm,amsmath}
\usepackage{enumerate}
\usepackage{cite}
\usepackage[hidelinks]{hyperref}
\usepackage{dsfont}
\usepackage{tikz}
\usepackage{lscape}

\def \Nats {\mathds{N}}

\def \restrictedto {\upharpoonright}
\newcommand{\set}[1]{\{#1\}}
\newcommand{\setarg}[2]{{\{#1\ |\ #2\}}}
\newcommand{\setcolon}[2]{{\{#1:#2\}}}

\newtheorem{lemma}{Lemma}[section]

\newtheorem{corollary}[lemma]{Corollary}
\newtheorem{prop}[lemma]{Proposition}

\theoremstyle{definition}
\newtheorem{definition}[lemma]{Definition}

\newtheorem{observation}[lemma]{Observation}
\newtheorem{example}[lemma]{Example}
\newtheorem*{remark}{Remark}

\theoremstyle{definition}
\newtheorem*{notation}{Notation}

\def \col {\mathfrak{c}}
\def \domcolor {\widetilde{\col}}
\def \descolor {\hat{\col}}
\newcommand{\pairs}[1]{[#1]^2}

\DeclareMathOperator{\CB}{CB}
\newcommand{\CBr}[1]{\CB(#1)}

\DeclareMathOperator{\tree}{T}
\newcommand{\subtree}[2]{\tree^{#1}(#2)}
\newcommand{\treeorder}{<^*}
\newcommand{\treedescendant}{\lhd^*}

\DeclareMathOperator{\Fan}{Fan}
\newcommand{\subfan}[1]{\Fan^{-}(#1)}
\newcommand{\fan}[1]{\Fan(#1)}

\newcommand{\ramseynotation}[4]{#1 \to_{cl} (#2)^{#3}_{#4}}
\newcommand{\closedramseynumber}[2]{R^{cl}(#1,#2)}

\DeclareMathOperator{\cnf}{CNF}
\newcommand{\cnfcut}[2]{\cnf_{#1}(#2)}
\DeclareMathOperator{\component}{\mathit{CC}}
\newcommand{\cc}[2]{\component_{#1}(#2)}

\DeclareMathOperator{\ordertype}{\mathit{ord}}
\newcommand{\otp}[1]{\ordertype(#1)}

\newcommand{\neighbours}[1]{N(#1)}
\newcommand{\neighboursInSet}[2]{N(#1)\cap #2}

\newcommand{\layer}[2]{L^{#1}_{#2}}
\newcommand{\layerTheta}[3]{L(#1)^{#2}_{#3}}

\newcommand{\domsep}[4]{\domcolor(#1,#2;#3,#4)}

\newcommand{\descol}[3]{\descolor(#1,#2,#3)}

\newcommand{\edges}[2]{E^{#1}_{#2}}

\newcommand{\lrg}[3]{\mathfrak{F}(#1)^{#2}_{#3}}

\newcommand{\ordPersFunc}[1]{\rho_{#1}}

\newcommand{\placeInFan}[1]{L(#1)}

\newcommand{\subskel}{\mathrel{\subseteq^{\rm sk}}}
\newcommand{\subskeln}[1]{\mathrel{\subseteq^{\rm sk}_{#1}}}

\newcommand{\oppress}{\mathrel{\perp}}
\newcommand{\harass}{\mathrel{_\omega{\perp}}}
\newcommand{\subcof}{\mathrel{\subseteq_{\rm cof}}}

\begin{document}

\begin{abstract}
We show that the \emph{closed ordinal Ramsey number} $\closedramseynumber{\omega\cdot 2}{3}$ is equal to $\omega^3\cdot 2$.
\end{abstract}

\title{Calculating the closed ordinal Ramsey number $\closedramseynumber{\omega\cdot 2}{3}$}
\author{Omer Mermelstein}
\address{Department of Mathematics\\
Ben-Gurion University of the Negev\\
Beer-Sheva 8410501, Israel}
\email{omermerm@math.bgu.ac.il}
\maketitle



\section{Introduction}

For a set of ordinals $J$, denote the order-type of $J$ by $\otp{J}$. For an ordinal $\alpha$, denote $[J]^\alpha = \setarg{X\subseteq J}{\otp{X} = \alpha}$. For a nonzero cardinal $\kappa$, a natural number $n$, and ordinals $\beta$ and $\alpha_i$ for all $i\in\kappa$, we write
\[
\ramseynotation{\beta}{\alpha_i}{n}{i\in\kappa}
\]
to mean that for every colouring $\col: [\beta]^n\to \kappa$ of subsets of $\beta$ of size $n$ in $\kappa$ many colours, there exist some $i\in\kappa$ and $X\subseteq \beta$ such that $\otp{X} = \alpha_i$, $X$ is closed in its supremum, and $[X]^n\subseteq \col^{-1}(\set{i})$. Should such an ordinal exist, let $R^{cl}(\alpha_i)^n_{i\in\kappa}$ denote the least ordinal $\beta$ for which the statement $\ramseynotation{\beta}{\alpha_i}{n}{i\in\kappa}$ holds. We call $R^{cl}(\alpha_i)^n_{i\in\kappa}$ the \emph{closed ordinal Ramsey number of $(\alpha_i)^n_{i\in\kappa}$}. When $n$ is omitted, by convention n=2.

For a history of partition relations and Rado's arrow notation see \cite{HajLar}. The ordinal partition calculus was introduced by Erd\H{o}s and Rado in \cite{ErdRad}. Topological partition calculus was considered by Baumgartner in \cite{Baum}. Baumgartner's work was continued in recent papers on topological (closed) ordinal partition relations by Caicedo, Hilton, and Pi\~{n}a \cite{Pina},\cite{HilPige},\cite{CaicedoHilton}.

Caicedo and Hilton proved recently that $\omega^2\cdot 3\leq \closedramseynumber{\omega\cdot 2}{3} \leq \omega^3\cdot 100$ \cite[Theorem 8.1]{CaicedoHilton} and provided also the upper bound $\closedramseynumber{\omega^2}{k}\leq \omega^\omega$ for every positive integer $k$ \cite[Theorem 7.1]{CaicedoHilton}. The lower bound $\closedramseynumber{\omega^2}{k} \geq \omega^{k+1}$ is a consequence of \cite[Theorem 3.1]{CaicedoHilton}.

In this paper we calculate the exact value $\closedramseynumber{\omega\cdot 2}{3}=\omega^3\cdot 2$. In a subsequent paper we will show that $\closedramseynumber{\omega^2}{3} = \omega^6$.

\subsection{Description of proof}
In order to prove $\closedramseynumber{\omega\cdot 2}{3}\leq \omega^3\cdot 2$, we must show that for each colouring $\col:\pairs{\omega^3\cdot 2}\to 2$ with no homogeneous triples of colour 1, there is some homogeneous $Y\in[\omega^3\cdot 2]^{\omega\cdot 2}$ of colour 0, closed in its supremum. However, the first part of the proof is the canonization of any finite pair-colouring of an ordinal less than $\omega^\omega$.

We show first that for each $\delta<\omega^\omega, k\in\Nats$ and pair-colouring $\col: \pairs{\delta}\to k$, there exists some $X\subseteq \delta$, closed in its supremum with $\otp{X} = \delta$, such that $\col\restrictedto \pairs{X}$ is a \emph{canonical} colouring (Definition \ref{canonicalDef}). A canonical colouring is one such that for ``most'' $\set{\alpha,\beta}\in\pairs{\delta}$, the value $\col(\set{\alpha,\beta})$ depends only on the Cantor-Bendixson ranks of the points $\alpha,\beta$ in the topological space $\delta$, and on the Cantor normal form of $\delta$. Canonical colourings will also be used in a future paper in the proof of $\closedramseynumber{\omega^2}{3} = \omega^6$.

The bound $\closedramseynumber{\omega\cdot 2}{3}\leq \omega^3\cdot 2$ is achieved by a combinatorial analysis of canonical colourings in two colours, and the following simple lemma on finite sets:

\begin{lemma}
\label{finiteSetLemma}
Whenever $F\subseteq Fin(\Nats)$ is an infinite $\subseteq$-antichain, there exists an infinite set $\setcolon{A_i}{i<\omega}\subseteq F$ and points $\setcolon{k_i}{i<\omega}\subseteq\Nats$ such that $k_i\in A_i\setminus\bigcup_{j\neq i} A_j$.
\end{lemma}

\begin{proof}
First, for every countably infinite collection $\setcolon{A_i}{i<\omega}$ of distinct finite subsets of $\Nats$, there is some infinite $I\subseteq\omega$ and sets $\setcolon{X_i}{i\in I}$ such that $A_i\cap A_j = X_i$ for all $i,j\in I$ such that $i<j$: we define sets $B_n$ and natural numbers $i_n$ such that
\begin{enumerate}
\item 
$B_0 = \omega$
\item
$i_n = \min B_n$
\item
$B_{n+1}\subseteq \setarg{i\in B_n}{i>i_n}$ is infinite and $A_{i_n}\cap A_j$ is constant for all $j\in B_{n+1}$.
\end{enumerate}
Point 3 is gotten by applying the pigeonhole principle, as $A_{i_n}$ is finite and $B_n$ is infinite. Thus $X_{i_n} = \bigcap_{k\geq n}A_{i_k}$, and hence $m<n\implies X_m\subseteq X_n$.

Given an infinite $\subseteq$-antichain $F\subseteq Fin(\Nats)$, we may assume by shrinking it, that it is as above for some enumeration $F=\set{A_i}_{i\in\omega}$. Let $Y_i = A_i\setminus X_i$. Since $F$ is an antichain, $Y_i\neq \emptyset$. Moreover, for distinct $i,j$ it holds that $Y_i\cap A_j \subseteq X_{\min\set{i,j}}\subseteq X_i$. This implies that $Y_i$ is disjoint from $A_j$ for every $i\neq j$. Now choose $k_i\in Y_i$ arbitrarily.
\end{proof}

The bound $\closedramseynumber{\omega\cdot 2}{3}\geq \omega^3\cdot 2$ is achieved by describing a single colouring $\col:\pairs{\omega^3\cdot 2}\to 2$ such that for each $\theta<\omega^3\cdot 2$, the colouring $\col\restrictedto\pairs{\theta}$ demonstrates $\closedramseynumber{\omega\cdot 2}{3}>\theta$.
\section{Preliminaries}

For every nonzero ordinal $\alpha$ there exist a unique nonzero $k\in\Nats$ and a unique sequence of ordinals $\beta_1 \geq ... \geq \beta_k$ such that
\[
\alpha = \omega^{\beta_1} + \omega^{\beta_2} + \dots + \omega^{\beta_k}.
\]
Also, there exist a unique $l\in\Nats$, a sequence of ordinals $\gamma_1 >\dots > \gamma_l$, and a sequence of nonzero natural numbers $m_1,\dots,m_l$ such that
\[
\alpha = \omega^{\gamma_1}\cdot m_1 + \omega^{\gamma_2}\cdot m_2 + \dots + \omega^{\gamma_l}\cdot m_l.
\]
These representations of $\alpha$ are two variations of the Cantor normal form of an ordinal $\alpha$. Here, we shall refer to the first as the \emph{Cantor normal form}. The \emph{Cantor-Bendixson rank\footnote{This is the Cantor-Bendixson rank of $\alpha$ as a point in the space of ordinals}} (CB rank) of $\alpha$ is $\beta_k$ (which is equal to $\gamma_l$), and is denoted by $\CBr{\alpha}$. We also denote $\placeInFan{\alpha} = m_l$. For the ordinal $\alpha=0$, see the following remark.

\begin{remark}
We wish to consider $\CBr{\alpha}$ and $\placeInFan{\alpha}$ for every ordinal $\alpha$, but these values do not naturally make sense for $0$. We solve this for $\alpha=0$ by either omitting $0$ completely, or arbitrarily setting $\CBr{0} = 0$ and $\placeInFan{n} = n+1$ for each $n\in\omega$, virtually omitting $0$.
\end{remark}

Beside the ordinary $\leq$ order relation on ordinals, we endow the ordinals with an anti-tree ordering $\treeorder\subseteq\leq$. By anti-tree, we mean that $\setarg{\beta}{\alpha\treeorder\beta}$ is well-ordered by $\treeorder$ for all $\alpha$. This is the same ordering $\treeorder$ presented in section 8 of \cite{CaicedoHilton}.

We say that $\beta\treeorder\alpha$ whenever $\alpha = \beta + \omega^\gamma$ for some nonzero ordinal $\gamma$ with $\gamma > \CBr{\beta}$. Equivalently, for some $\gamma > \CBr{\beta}$, $\alpha$ is the least ordinal of $CB$ rank $\gamma$ with $\beta\leq \alpha$. We write $\beta\treedescendant\alpha$ if $\alpha$ is the unique immediate successor of $\beta$ in $\treeorder$. A graphical representation of $\treeorder$ on $\omega^3\cdot 2$ can be found in Figure \ref{figure} at the end of the paper.

\begin{notation}
Let $\alpha$ be some ordinal.
\begin{itemize}
\item
Denote $\subtree{}{\alpha} = \set{\alpha}\cup\setarg{\beta}{\beta\treeorder\alpha}$ and $\subtree{=n}{\alpha}=\setarg{\beta\treeorder\alpha}{\CBr{\beta}=n}$.
\item
Assuming $\CBr{\alpha} = \gamma+1$ is a successor ordinal, denote
\[
\subfan{\alpha} = \setarg{\beta}{\beta\treedescendant\alpha}.
\]
Equivalently, $\subfan{\alpha} = \setarg{\alpha + \omega^{\gamma}\cdot (i+1)}{i\in\omega}$. Letting $\widehat{\alpha}$ be the unique ordinal such that $\alpha\treedescendant\widehat{\alpha}$, denote
\[
\fan{\alpha} = \subfan{\widehat{\alpha}},
\]
the unique set of the form $\subfan{\beta}$ of which $\alpha$ is a member.
\end{itemize}
Let $I$ be a set of ordinals.
\begin{itemize}
\item
Denote by $\ordPersFunc{I}$ the unique order preserving bijection from $I$ onto $\otp{I}$.
\item
Write $J\subcof I$ to mean that $J$ is a cofinal subset of $I$.
\end{itemize}
\end{notation}

From here onwards, \textbf{all ordinals are assumed to be smaller than $\omega^\omega$}.

\subsection{Skeletons}

The following is the main notion we use in thinning out arguments.

\begin{definition}
\label{skeletonDefinition}
	Let $\delta<\omega^\omega$ be an ordinal. We say that $I\subseteq \delta$ is a \emph{skeleton} of $\delta$ if:
	\begin{enumerate}[(S1)]
	\item
	$I$ is closed in $\delta$ in the order topology;
	\item
	$\otp{I} = \delta$;
	\item
	$\alpha \treeorder \beta$ if and and only if $\ordPersFunc{I}(\alpha) \treeorder \ordPersFunc{I}(\beta)$, for all $\alpha,\beta\in I$.
	\end{enumerate}
	For an arbitrary set of ordinals $J$, we say that $I\subseteq J$ is a skeleton of $J$ if $\ordPersFunc{J}[I]$ is a skeleton of $\otp{J}$. Write $I\subskel J$.
	\end{definition}
	
\begin{definition}
For each $n\in\Nats$, write $I\subskeln{n} J$ and say that $I$ is an \emph{$n$-skeleton} of $J$ if $I\subskel J$ and
\[
\setarg{\alpha\in J}{\fan{\alpha}\cap I\neq\emptyset~\text{and}~\placeInFan{\ordPersFunc{J}(\alpha)}\leq n}\subseteq I.
\]
Informally, if a fan was not removed entirely in the transition from $J$ to $I$, then its first $n$ elements were not removed. Note that the relation $\subskeln{0}$ is in fact $\subskel$ and that $\subskeln{m}$ implies $\subskeln{n}$ whenever $m>n$.
\end{definition}

\begin{observation}
For every natural $n$, the relation $\subskeln{n}$ is transitive.
\end{observation}

\begin{definition}
	Let $\delta$ be an ordinal with Cantor normal form
	\[
	\delta = \omega^{\gamma_1} + \omega^{\gamma_2} + \dots + \omega^{\gamma_k}.
	\]
	For $\beta<\delta$ define $\cnfcut{\delta}{\beta}$ to be the least $i$ such that $\beta\leq \omega^{\gamma_1} + \dots + \omega^{\gamma_i}$.
	Define $\cc{\delta}{i} = \setarg{\beta < \delta}{\cnfcut{\delta}{\beta} = i}$. Explicitly, $\cc{\delta}{i}$ is the collection of $\beta<\delta$ such that $\sum_{j=1}^{i-1}\omega^{\gamma_j}<\beta\leq \sum_{j=1}^{i}\omega^{\gamma_j}$.
	Define $k_\delta = k$, the number of non-empty components in this decomposition of $\delta$.
\end{definition}

The following observation is a straightforward unravelling of the definitions. Note that, by definition, $\ordPersFunc{I}$ is a $\treeorder$-isomorphism.

\begin{observation}
If $I\subskel\delta$ for an ordinal $\delta<\omega^\omega$, then for every $\alpha\in I$ we have $\cnfcut{\delta}{\alpha} = \cnfcut{\delta}{\ordPersFunc{I}(\alpha)}$ and $\CBr{\alpha} = \CBr{\ordPersFunc{I}(\alpha)}$.
\end{observation}

For the remainder of this section, fix some ordinal $\delta < \omega^\omega$.

\begin{observation}
The collection $\setarg{\cc{\delta}{i}}{1\leq i\leq k_\delta}$ is a partition of $\delta$ into $k_\delta$ closed subsets of $\delta$. Moreover, $\cc{\delta}{k_\delta}$ is of order type $\omega^{\gamma_{k_{\delta}}}$, and for every $i<k_\delta$, the set $\cc{\delta}{i}$ is either a singleton (when $\gamma_i = 0$) or of order type $\omega^{\gamma_i} + 1$.
\end{observation}

The following lemma follows directly from the definitions and the uniqueness of the Cantor normal form.

\begin{lemma}
\label{skeletonIfComponents}
Let $I\subseteq\delta$. Then $I\subskeln{n}\delta$ if and only if $I\cap \cc{\delta}{i}\subskeln{n}\cc{\delta}{i}$ for every $1\leq i\leq k_\delta$. \qed
\end{lemma}

The following is a way of ensuring an infinite successive thinning out to skeletons is a skeleton.

\begin{definition}
\label{fanPreservingDef}
Let $\gamma$ be some ordinal. For each $i<\gamma$, let $I_i\subseteq\delta$. We say that the collection $\setcolon{I_i}{i<\gamma}$ is \emph{fan preserving} if whenever $\alpha\leq\gamma$ is a limit ordinal and $\beta < \delta$ is such that $\fan{\beta}\cap I_{i}$ is infinite for all $i < \alpha$, then $\fan{\beta}\cap \bigcap_{i < \alpha} I_{i}$ is infinite.
\end{definition}

\begin{lemma}
\label{intersectionOfFanPreserving}
Let $k$ be some positive integer. Let $\setcolon{I_i}{i<\gamma}$ be a fan preserving collection of $n$-skeletons of $\omega^k + 1$. Then $\mathbb{I} = \bigcap_{i<\gamma} I_i$ is an $n$-skeleton of $\omega^k + 1$.
\end{lemma}

\begin{proof}
By induction on $k$. Every skeleton of $\omega^{k}+1$ must contain the point $\omega^{k}$ and so $\omega^{k}\in \mathbb{I}$. Additionally, every skeleton intersects infinitely with $\subfan{\omega^{k}}$. By fan preservation, $\subfan{\omega^{k}}\cap\mathbb{I}$ is infinite, and contains the first $n$ elements of $\subfan{\omega^k}$.

Enumerate $\subfan{\omega^{k}}\cap\mathbb{I} = \setcolon{\beta_i}{i<\omega}$. For every $\alpha<\gamma$, we have $\subtree{}{\beta_i}\cap I_{\alpha}\subskeln{n}\subtree{}{\beta_i}$. The collection $\set{\subtree{}{\beta_i}\cap I_\alpha}_{\alpha<\gamma}$ of $n$-skeletons of $\subtree{}{\beta_i}\cong\omega^{k-1} +1$ is fan preserving. By the induction hypothesis $\subtree{}{\beta_i}\cap \mathbb{I}$ is an $n$-skeleton of $\subtree{}{\beta_i}$. Thus,
\[
\mathbb{I} = \bigcup_{i\in\omega} (\subtree{}{\beta_i}\cap \mathbb{I})\cup\set{\omega^{k}}
\]
is an $n$-skeleton of $\omega^{k}+1$.
\end{proof}

\begin{corollary}
\label{fanPreservingCorOmega}
Let $\gamma$ be some ordinal. Whenever $\setcolon{I_i}{i<\gamma}$ is a fan preserving collection of $n$-skeletons of $\delta$, then $\mathbb{I} = \bigcap_{i<\gamma} I_i$ is an $n$-skeleton of $\delta$. \qed
\end{corollary}

\begin{lemma}
\label{goodIncreasingImpliesFanPreserving}
Let $\setcolon{I_i}{i<\gamma}$, for some ordinal $\gamma$, and assume that there is some $\sigma:\gamma\to \Nats$ with finite fibers $($that is $|\sigma^{-1}[\set{n}]|<\infty$ for each $n\in\Nats)$ such that $I_{\alpha}\subskeln{\sigma(\alpha)}\bigcap_{i<\alpha} I_i$ for all $i<\gamma$. Then $\setcolon{I_i}{i<\gamma}$ is a fan preserving collection of skeletons of $I_0$.
\end{lemma}

\begin{proof}
By induction on $\gamma$. Let $\alpha\leq \gamma$ be a limit ordinal such that $\set{I_i}_{i<j}$ is fan preserving for every $j<\alpha$. By Corollary \ref{fanPreservingCorOmega}, $I_{j}\subskel I_0$ for every $j<\alpha$.

Let $\beta$ be such that $\fan{\beta}\cap I_i$ is infinite for all $i<\alpha$. For an arbitrary $n\in\Nats$, let $i_n<\alpha$ be maximal such that $\sigma(i_n)\leq n$. By an induction argument on $j$, $I_j\subskeln{n} I_{i_n}$ for every $i_n<j<\alpha$. Thus, $|\fan{\beta}\cap I_\alpha| \geq n$ for any $n\in\Nats$.
\end{proof}

\subsection{The sets $F(\omega^k)^r_n$}
We will thin out colourings using families of skeletons as in the statement of Lemma \ref{goodIncreasingImpliesFanPreserving}. This compromise prevents us from controlling the colouring on entire sets of the form $\subtree{=n}{\alpha}$, but it allows us to control it on ``large'' sets of a specific form.

Let $k,r\in\Nats$. We recursively define a set $F(\omega^{k})^r_n\subseteq\subtree{=n}{\omega^k}$ for each $n\leq k$ (see Figure \ref{largeSetsFigure})
	\[
	F(\omega^k)^r_n =\begin{cases}
	\set{\omega^k} &,n = k
	\\
	\bigcup_{\beta\in F(\omega^k)^r_{n+1}}\setarg{\gamma\in\subfan{\beta}}{\placeInFan{\gamma} > r} &, n<k
	\end{cases}
	\]
An equivalent non-recursive definition, when $n<k$, is:
	\[
		F(\omega^k)^r_n = \setarg{\beta\in\subtree{=n}{\omega^k}}{\min\setcolon{\placeInFan{\beta'}}{\beta\treeorder\beta'\treeorder \omega^k} > r, \placeInFan{\beta} > r}.
	\]
	Note that $F(\omega^k)^0_n = \subtree{=n}{\omega^k}$.
	
\begin{figure}
\caption{The sets $F(\omega^3)^1_n$}
\label{largeSetsFigure}
\begin{tikzpicture}[x=8,y=30]
\foreach \i in{0}
{
	\node[fill=black,circle, inner sep=2.5pt] at (\i,4) {};
	\foreach \j in{-12,0,12}
	{
		\node[fill=black,circle, inner sep=2pt] at (\i+\j,2.5) {};
		\draw (\i,4) -- (\i+\j,2.5);
		\foreach \k in{-4,0,4}
			{
			\node[fill=black,circle, inner sep=1pt] at (\i+\j +\k,1) {};
			\draw (\i+\j,2.5) -- (\i+\j +\k,1);
			\foreach \l in{-1,0,1}
				{
				\node[fill=black,circle, inner sep=0.8pt] at (\i+\j+\k+\l,0) {};
				\draw (\i+\j+\k,1) -- (\i+\j +\k+\l,0);
				}
			\node[label=right:\tiny $...$] at (\i+\j+\k+1/2,0) {};
			}
		\node[label=right:$\dots$] at (\i+\j+3.5,1) {};
	}
	\node[label=right:\Huge$\dots$] at (\i+12,2.5) {};
}
\draw[thick,rounded corners=8pt] (-1,4-8/30) rectangle (1,4+8/30);

\draw[thick,rounded corners=5pt] (-1,2.5-8/45) rectangle (16, 2.5+8/45);

\draw[semithick,rounded corners] (-0.75,1-8/60) rectangle (6, 1+8/60);
\draw[semithick,rounded corners] (11.25,1-8/60) rectangle (18, 1+8/60);

\foreach \r in{0,1}
{
\draw[rounded corners=2.5pt] (12*\r-0.4,-8/90) rectangle (12*\r+2.3, 8/90);
\draw[rounded corners=2.5pt] (12*\r+3.6,-8/90) rectangle (12*\r+6.3, 8/90);
}
\node at (22,4) {$F(\omega^3)^1_3$};
\node at (22,2.5) {$F(\omega^3)^1_2$};
\node at (22,1) {$F(\omega^3)^1_1$};
\node at (22,0) {$F(\omega^3)^1_0$};
\end{tikzpicture}
\end{figure}

We extend the definition to an arbitrary $\alpha<\omega^\omega$. Denote $k=\CBr{\alpha}$. Observe that $\subtree{}{\alpha}\cong \omega^k +1$ and define
\[
F(\alpha)^r_n = \ordPersFunc{\subtree{}{\alpha}}^{-1}[F(\omega^k)^r_n]
\]

\begin{definition}
Fix some $\alpha<\omega^\omega$ and $n\leq \CBr{\alpha}$. Define
\begin{gather*}
\lrg{\alpha}{r}{n} = \setarg{A\subseteq\subtree{=n}{\alpha}}{F(\alpha)^r_n\subseteq A}
\\
\lrg{\alpha}{}{n} = \bigcup_{r\in\Nats} \lrg{\alpha}{r}{n}.
\end{gather*}
\end{definition}

\begin{example}
As an instructive example, the set
\[
\bigcup_{k\in\omega}\setarg{\omega\cdot k +l}{l > k}
\]
is \textbf{not} an element of $\lrg{\omega^2}{}{0}$.
\end{example}

\begin{observation}
For a fixed $n$, the set $\lrg{\alpha}{}{n}$ is closed under finite intersections. This is due to the fact $\lrg{\alpha}{r}{n}\subseteq\lrg{\alpha}{s}{n}$, whenever $r<s$. The set $\lrg{\alpha}{}{n}$ is the filter on $\subtree{=n}{\alpha}$ generated by $\set{F(\alpha)^r_n}_{r\in\Nats}$.
\end{observation}

\begin{observation}
	Let $\delta<\omega^\omega$ be an ordinal, let $I\subskel \delta$ and let $\alpha\in I$. Then $\ordPersFunc{I}[F\cap I]\in\lrg{\ordPersFunc{I}(\alpha)}{r}{n}$, for any $F\in\lrg{\alpha}{r}{n}$.
\end{observation}

\begin{observation}
\label{middleLayerLarge}
Let $\theta<\omega^\omega$ be some ordinal and let $A\subseteq\subtree{=n}{\theta}$ for some $n\leq\CBr{\theta}$. Then for any $n\leq k\leq \CBr{\theta}$, $A\in\lrg{\theta}{r}{n}$ if and only if
\[
\setarg{\beta\in\subtree{=k}{\theta}}{\subtree{}{\beta}\cap A\in\lrg{\beta}{r}{n}}\in\lrg{\theta}{r}{k}
\]
\end{observation}

\begin{corollary}
\label{largeSetsContainSubtree}
Let $\theta$ be some ordinal and let $l_1 < \dots < l_k<\CBr{\theta}$ and $A_1,\dots,A_k$ be such that $A_i\in\lrg{\theta}{}{l_i}$. Then there exists some $X\subseteq \bigcup_{n=1}^k A_n$ closed in its supremum with $\otp{X}= \omega^k$. Moreover, $X$ can be chosen such that $\ordPersFunc{X}[X\cap F]\in\lrg{\omega^k}{}{i}$ for any $1\leq i\leq k$ and $F\in\lrg{\theta}{}{l_i}$.
\end{corollary}

\begin{proof}
By induction on $k$. Choose some $B=\setcolon{b_r}{r<\omega}$ with $\otp{B}=\omega$ such that $b_r\in F(\theta)^r_{l_k}$. Let
\[
C_i = \setarg{\beta\in \subtree{=l_k}{\theta}}{\subtree{}{\beta}\cap A_i\in\lrg{\beta}{}{l_i}}.
\]
By the above lemma $C_i\in\lrg{\theta}{}{k}$, and so we have that $A_k\cap\bigcap_{i=1}^{k-1} C_i\in\lrg{\theta}{}{l_k}$. Denote this set by $\mathbb{A}_k$. For each $\alpha\in\mathbb{A}_k$, by induction hypothesis, choose some $X_\alpha\subseteq\bigcup_{i=1}^{k-1}(\subtree{}{\alpha}\cap A_i)$ closed in its supremum with $\otp{X_{\alpha}}= \omega^{k-1}$. Then $X = \bigcup_{\alpha\in\mathbb{A}\cap B} X_{\alpha}$ is closed in its supremum with $\otp{X}= \omega^k$.
\end{proof}

\section{Reducing a finite pair-colouring to a canonical colouring}
For this section fix some ordinal $\delta<\omega^\omega$ and some positive integer $n$.
\begin{definition}
Let $\col:\delta\to n$ be some colouring. For $r\in\Nats$, say that $\col$ is \emph{$r$-good} if for every $\theta<\delta$ and $l<\CBr{\theta}$, there is some colour $c<n$ such that 
\[
\setarg{\alpha\in \subtree{=l}{\theta}}{\col(\alpha) = c}\in\lrg{\theta}{r}{l}.
\]
Say that $I\subskel \delta$ is \emph{$r$-good with respect to $\col$}, if $\col\circ\ordPersFunc{I}^{-1}$ is $r$-good.

\end{definition}

\begin{lemma}
	\label{firstLemma}
	Let $\col:\delta \to n$ be some colouring. Then for any $r\in\Nats$ there exists some $I\subskeln{r} \delta$ $r$-good with respect to $\col$.
\end{lemma}

\begin{proof}
	By Lemma \ref{skeletonIfComponents}, we may assume $\delta = \omega^{k+1}+1$ for some $k\in\Nats$. We prove by induction on $k$.
	
	For $i\in \Nats$, denote $T_i: = \subtree{}{\omega^k\cdot (i+1)}$. Considering $\col\restrictedto T_i$ as a colouring of $\omega^k + 1$, by the induction hypothesis there is some $r$-good $S_i\subskeln{r} T_i$. Since $S_i$ is $r$-good, let $f_i: k+1 \to n$ be the function mapping $m$ to the colour assigned to a large set of $\alpha\in S_i$ with $\CBr{\alpha} = m$. By the pigeonhole principle, there is an infinite subset $B\subseteq\Nats$ such that $f_i = f_j$ for all $i,j\in B$. Let $B' = B\cup \Nats^{<r}$ and take $I = \{\omega^{k+1}\}\cup \bigcup_{i\in B'} S_i$. By construction we have that $I\subskeln{r} \delta$ and $r$-goodness follows by Observation \ref{middleLayerLarge}.
\end{proof}

\begin{definition}
\label{normalDef}
	Let $\col:\pairs{\delta}\to n$ be some colouring.
	Say that $\col$ is \emph{normal} if whenever $\beta_1,\beta_2<\delta$, $\beta_1\treeorder\beta_2$, the value $\col(\set{\beta_1,\beta_2})$ depends only on $\CBr{\beta_1}$, $\CBr{\beta_2}$ and $\cnfcut{\delta}{\beta_2}$. Namely, there is a function $\descolor$, independent of $\beta_1,\beta_2$, such that
		\[
		\col(\set{\beta_1,\beta_2}) = \descolor(\cnfcut{\delta}{\beta_2},\CBr{\beta_2},\CBr{\beta_1}).
		\] 
	Say that $I\subskel \delta$ is \emph{normal} with respect to $\col$ if $\col\circ\ordPersFunc{I}^{-1}$ is normal.
\end{definition}

\begin{notation}
For a pair-colouring $\col:\pairs{X}\to n$ and an element $\alpha\in X$, write $\col_\alpha$ for any colouring $\col_\alpha:X\to n$ assigning $\col_\alpha(\beta) = \col(\set{\alpha,\beta})$ to every $\beta\in X\setminus\set{\alpha}$.
\end{notation}

\begin{lemma}
	\label{descnedentColouringDependsOnCBRankSkeleton}
	Let $\col: \pairs{\delta} \to n$ be some pair-colouring. Then there exists $I\subskel \delta$ normal with respect to $\col$.
\end{lemma}

\begin{proof}
	By Lemma \ref{skeletonIfComponents}, it suffices to prove for $\delta = \omega^k+1$ for $k\in\Nats$. We thin out $\omega^k + 1$ inductively.
	
	Assume that for every $\alpha$ with $\CBr{\alpha} = m$, there exists a function $f_\alpha$ such that for every $\beta_1,\beta_2\in\subtree{}{\alpha}$ with $\beta_2\treeorder\beta_1$, we have $\col(\set{\beta_1, \beta_2}) = f_\alpha(\CBr{\beta_1},\CBr{\beta_2})$. We replace $\omega^k + 1$ with a skeleton $J$, such that the assumption holds for $m+1$ with respect to $\col\restrictedto\pairs{J}$.
	
	Let $\alpha$ be an ordinal with $\CBr{\alpha} = m+1$ and enumerate $\subfan{\alpha} = \setcolon{\gamma_i}{i<\omega}$. By the pigeonhole principle, we may assume that $f_{\gamma_i}$ is constant with $i$, and denote it by $f$. Consider the colouring $\col_\alpha:\subtree{}{\alpha}\to n$ and choose some $0$-good $I_\alpha\subskel \subtree{}{\alpha}$ as in Lemma \ref{firstLemma}. Let $g:m \to n$ be the function defined $g(i) = \col_\alpha(\beta)$, where $\beta$ is any $\beta\in I_\alpha$ with $\CBr{\beta} = i$. So for every $\beta_1,\beta_2\in I$ with $\beta_2\treeorder\beta_1$ we have
	\[
	\col(\set{\beta_1,\beta_2}) =
	\begin{cases}
	g(\CBr{\beta_2}), &\beta_1 = \alpha
	\\
	f(\CBr{\beta_1},\CBr{\beta_2}), &\beta_1 \neq \alpha
	\end{cases}
	\]
	Replacing $\subtree{}{\alpha}$ with $I_\alpha$ for each $\alpha$ with $\CBr{\alpha} = m+1$ gives us a skeleton as we seek.
	
	Iterating this process until the stage $m=k$ leaves us with $I\subskel\omega^k + 1$ as described in the statement of the Lemma.
\end{proof}

\begin{definition}
\label{goodnessDef}
	Let $\col:\pairs{\delta}\to n$ be some pair-colouring. For $\alpha<\delta$ and $r\in\Nats$, if $\col_\alpha$ is $r$-good, say that $\alpha$ is \emph{$r$-good for $\col$}. Say that $\alpha$ is \emph{good for $\col$} whenever $\alpha$ is $r$-good for $\col$ for some natural $r$.
	\\
	For $I\subskel\delta$, $\alpha\in I$, say that $\alpha$ is \emph{($r$-)good with respect to $I$ for $\col$} if $\ordPersFunc{I}(\alpha)$ is ($r$-)good for $\col_\alpha\circ\ordPersFunc{I}^{-1}$.
\end{definition}
 
\begin{observation}
\label{goodnessGoesToSkeleton}
If $J\subskel I\subskel \delta$ and $\alpha\in J$ is $r$-good with respect to $I$ for $\col$, then $\alpha$ is $r$-good with respect to $J$ for $\col$.
\end{observation}

\begin{lemma}
	\label{goodLemma}
	Let $\col:\pairs{\delta} \to n$ be some pair-colouring. Then there exists some $I\subskel \delta$ such that every $\alpha\in I$ is good with respect to $I$ for $\col$.
\end{lemma}

\begin{proof}
	Fix some injection $\sigma:\delta \to \omega$. We construct inductively $\setcolon{I_{\alpha}}{\alpha<\delta}$ such that
	\begin{enumerate}
	\item
	$I_0=\delta$
	\item
	$I_{\alpha}\subskeln{\sigma(\alpha)} \bigcap_{\beta<\alpha} I_\beta$;
	\item
	$\alpha$ is good with respect to $I_{\alpha}$ for $\col$.
	\end{enumerate}
	By Lemma \ref{goodIncreasingImpliesFanPreserving} and Corollary \ref{fanPreservingCorOmega}, $I = \bigcap_{\alpha<\delta} I_\alpha$ will be a skeleton as we seek.
	
	\smallskip
	Let $\alpha<\delta$ and assume that $I_{\gamma}$ was constructed as above for all $\gamma<\alpha$. Let $J_{\alpha} = \bigcap_{\gamma<\alpha} I_{\gamma}$, a skeleton of $\delta$ by Lemma \ref{goodIncreasingImpliesFanPreserving} and Corollary \ref{fanPreservingCorOmega}. Using Lemma \ref{firstLemma}, let $I_{\alpha}\subskeln{\sigma(\alpha)} J_{\alpha}$ be $\sigma(\alpha)$-good with respect to $\col_\alpha$.
\end{proof}

\begin{definition}
	Let $\col:\pairs{\delta}\to n$ be some pair-colouring. Let $\alpha$ be good for $\col$ and denote $i_\alpha = \cnfcut{\delta}{\alpha}$.
	\begin{itemize}
	\item
	For every $i_\alpha\neq m\leq k_\delta$, let $\epsilon_m = \sup\cc{\delta}{m}$. For each $m\neq i_\alpha$ and $l<\CBr{\epsilon_m}$, define $t_{\alpha}(m,l)$ to be the unique colour for which
	\[
	\setarg{\beta\in \subtree{=l}{\epsilon_m}}{\col(\set{\alpha,\beta}) = t_{\alpha}(m,l)}\in\lrg{\epsilon_m}{}{l}
	\]
	\end{itemize}
	Call $t_{\alpha}$ the \emph{goodness function} of $\alpha$ for $\col$.
\end{definition}

\begin{definition}
	Let $\col:\pairs{\delta}\to n$ be some pair-colouring. We say that $\col$ is \emph{uniformly good} if
	\begin{itemize}
		\item
		Every $\alpha\in I$ is good with respect to $I$ for $\col$.
		\item
		For any $\alpha<\delta$, the goodness function of $\alpha$ depends only on $\cnfcut{\delta}{\alpha}$ and $\CBr{\alpha}$. Namely, there is a function $\domcolor$, independent of $\alpha$, such that
		\begin{gather*}
		t_{\alpha} = \domcolor(\cnfcut{\delta}{\alpha},\CBr{\alpha})
		\end{gather*}
	\end{itemize}
	For $I\subskel \delta$, say that $I$ is \emph{uniformly good with respect to $\col$} if $\col\circ\ordPersFunc{I}^{-1}$ is uniformly good.
\end{definition}

\begin{notation}
	For a function $\domcolor$ as in the above definition we abuse notation and write 	
	\[
	\domsep{i_1}{j_1}{i_2}{j_2} := t_{\alpha}(i_2,j_2)
	\]
	where $\alpha<\delta$ is such that $\cnfcut{\delta}{\alpha} = i_1$ and $\CBr{\alpha}=j_1$.

\end{notation}

\begin{definition}
\label{canonicalDef}
	Let $\col:\pairs{\delta}\to n$ be some pair-colouring. We say that $I\subskel \delta$ is \emph{canonical} with respect to $\col$ if $I$ is normal and uniformly good with respect to $\col$.
	\\
	We say that $\col$ is \emph{canonical}, if $\delta$ itself is canonical with respect to $\col$.
\end{definition}

\begin{prop}
\label{canonicalSkeleton}
	Let $\col:\pairs{\delta} \to n$ be some pair-colouring. Then there exists some $I\subskel \delta$ canonical with respect to $\col$.
\end{prop}

\begin{proof}
	Let $J\subseteq\delta$ be as guaranteed by Lemma \ref{goodLemma}.
	
	Colour each element of $J$ by its goodness function with respect to $J$ for $\col$. Exercising Lemma \ref{firstLemma}, choose some $J'$ such that $t_{\alpha}$ depends only on $\cnfcut{\delta}{\alpha}$ and $\CBr{\alpha}$. Now use Lemma \ref{descnedentColouringDependsOnCBRankSkeleton} on $J'$ to receive a skeleton $I$.
\end{proof}

For the purpose of calculating bounds on closed ordinal Ramsey numbers, the above proposition allows us to assume every pair-colouring of $\delta$ in finitely many colours is canonical.

\section{Colouring in two colours}
For this section fix some ordinal $\delta<\omega^\omega$.

There is a natural one-to-one correspondence between pair-colourings $\col:\pairs{\delta}\to 2$ and graphs $G=(\delta, E)$, given by the rule $(\alpha,\beta)\in E \iff \col(\set{\alpha,\beta}) = 1$.
We say that a graph on $\delta$ is \emph{canonical} if its corresponding colouring is canonical.

Let $G=(V,E)$ be a graph. For $v\in V$, denote $\neighbours{v} = \setarg{u\in V}{(v,u)\in E}$. For $U\subseteq V$, denote $\neighbours{U} = \bigcup_{v\in U}\neighbours{v}$. We say that a set of vertices $U\subseteq V$ is a \emph{clique} if $\pairs{U}\subseteq E$ and that it is \emph{independent} if $\pairs{U}\cap E = \emptyset$.

When the context is clear, we use graph and colouring terminology interchangeably. When colouring with colours $\set{0,1}$, we will always think of $1$ as the edge colour and of $0$ as the non-edge colour.

\begin{definition}
	Let $G = (\delta,E)$ be some graph on $\delta$. Let $A,B\subseteq \delta$ be infinite, disjoint and without maxima.
	\begin{itemize}
	\item
	Write $A\oppress B$ to mean that for all $X$, if $X\subcof A$, then $B\setminus\neighbours{X}$ is finite.
	\item
	Write $A\harass B$ if $A\oppress B$ and in addition $\neighboursInSet{a}{B}$ is finite for every $a\in A$.
	\end{itemize}
\end{definition}

Note that if $A\oppress B$ then $A_0\oppress B_0$ for every $A_0\subcof A$ and $B_0\subseteq B$.

\begin{lemma}
	\label{harassmentLemma}
	Let $G =(\delta,E)$. Let $A,B\subseteq \delta$ be such that $A\harass B$. Then there is some $A_0\subcof A$ and some $B_0\subseteq B$, cofinite in $B$, such that $\neighboursInSet{b}{A_0}$ is cofinite in $A_0$, for all $b\in B_0$.
\end{lemma}

\begin{proof}
	We may assume $\otp{A} = \omega$ by thinning out to a cofinal subset. We thin out $A$ so either $\neighboursInSet{a}{B}$ is constant for all $a\in A$, or $\neighboursInSet{a}{B}\neq \neighboursInSet{b}{B}$ for all distinct $a,b\in A$. Since $A\oppress B$, the latter must hold. Consider $F = \setarg{\neighboursInSet{a}{B}}{a\in A}$ as a partially ordered set under inclusion. By Ramsey's theorem, in $(F,\subseteq)$ there exists either an infinite chain or an infinite antichain.
	
	Suppose first that there exists some infinite $\subseteq$-antichain $C\subseteq F$. By Lemma \ref{finiteSetLemma} we choose $\set{a_i}_{i\in\omega}\subseteq A$ and $\set{b_i}_{i\in\omega}\subseteq B$ such that $b_i\in\neighbours{a_i}\setminus\neighbours{\set{a_j}_{j\neq i}}$. Then for $X=\{a_{{2i}}\}_{i\in\omega}$, the set $B\setminus\neighbours{X}$ is infinite, contradicting $A\oppress B$. Thus, there is some infinite $A_0\subseteq A$ such that $\setarg{\neighboursInSet{a}{B}}{a\in A_0}$ is a chain in $(F,\subseteq)$. Let $B_0 = \neighboursInSet{A_0}{B}$, which by $A\oppress B$ is cofinite in $B$. For every $b\in B_0$, since $b$ is an element of the well-ordered ascending union $\bigcup_{a\in A_0}\neighboursInSet{a}{B}$, it must be that $b\in\neighbours{a}$ for cofinitely many $a\in A_0$.
\end{proof}

\subsection{Canonical triangle-free graphs}

We say that a graph $G=(V,E)$ is triangle-free if $\neighbours{v}$ is independent for every $v\in V$. For the remainder of this section, fix some canonical triangle-free $G=(\delta,E)$ with corresponding colouring $\col:\pairs{\delta}\to 2$.

\begin{notation}
For every $i\leq k_\delta$ and $j\leq \CBr{\sup\cc{\delta}{i}}$ denote
\[
\layer{i}{j} = \setarg{\alpha\in\cc{\delta}{i}}{\CBr{\alpha} = j}
\]
\end{notation}

\begin{lemma}
\label{edgeColourScarce}
If there exists no independent $X\subseteq \delta$ closed in its supremum with $\otp{X} = \omega^2$, then the following statements hold
\begin{enumerate}
\item
For fixed $i,j$, there is at most one $l$ such that $\descol{i}{j}{l} = 1$.
\item
For fixed $i_1,i_2,j$, there is at most one $l$ such that $\domsep{i_1}{j}{i_2}{l} = 1$.
\item
For fixed $i_1,i_2,l$, there is at most one $j$ such that $\domsep{i_1}{j}{i_2}{l} = 1$.
\end{enumerate}
\end{lemma}

\begin{proof}
Proof of (1): 
Assume to the contrary that $l_1<l_2$ are such that $\descol{i}{j}{l_r}=1$. Let $\alpha\in\layer{i}{j}$. Then $\subtree{=l_1}{\alpha}\cup\subtree{=l_2}{\alpha}\subseteq \neighbours{\alpha}$. Thus, the set $\subtree{=l_1}{\alpha}\cup\subtree{=l_2}{\alpha}$ is independent, and by Corollary \ref{largeSetsContainSubtree} contains a copy of $\omega^2$.

Proof of (2):
Assume to the contrary that $l_1<l_2$ are such that $\domsep{i_1}{j}{i_2}{l_r} = 1$. Choose arbitrarily some $\alpha\in\layer{i_1}{j}$ and denote $\theta = \sup\cc{\delta}{i_2}$. Let $F_r = \neighbours{\alpha}\cap\subtree{=l_r}{\theta}$ and note $F_r\in\lrg{\gamma}{}{l_r}$. Then $F_1\cup F_2$ is an independent set, and by Corollary \ref{largeSetsContainSubtree} contains a copy of $\omega^2$.

Proof of (3):
Assume to the contrary that $j_1<j_2$ are such that $\domsep{i_1}{j_r}{i_2}{l}=1$. Denote $\theta_r = \sup\cc{\delta}{i_r}$. For any $\alpha_1,\alpha_2\in \subtree{=j_1}{\theta_1}\cup \subtree{=j_2}{\theta_1}$, by assumption $\neighbours{\alpha_1}\cap \neighbours{\alpha_2}\cap \subtree{=l}{\theta_2}\in\lrg{\theta_2}{}{l}$. So by triangle-freeness $\set{\alpha_1,\alpha_2}\notin E$. Thus, $\subtree{=j_1}{\theta_1}\cup \subtree{=j_2}{\theta_1}$ is an independent set, which contains a copy of $\omega^2$ by Corollary \ref{largeSetsContainSubtree}.
\end{proof}

\begin{lemma}
\label{oppressImpliesHarass}
Let $B\in\lrg{\gamma}{}{n}$ for some $\gamma\leq \delta$. Then there exists some $B_0\subcof B$ such that whenever $A\subseteq\delta$ is such that $A\oppress B$ and $\neighboursInSet{a}{B}\notin\lrg{\gamma}{}{n}$ for each $a\in A$, it is also the case that $A\harass B_0$.
\end{lemma}

\begin{proof}
For each $r\in \omega$, choose arbitrarily some $\beta_r\in F(\gamma)^r_n$ and let $B_0 = \setcolon{\beta_r}{r<\omega}$. Let $A$ be as in the statement. By canonicity, for any $\alpha\in A$ we have $B\setminus\neighbours{\alpha}\in\lrg{\gamma}{s}{n}$ for some $s$ and therefore $\beta_r\notin \neighbours{\alpha}$ for all $r>s$. Thus, $\neighboursInSet{\alpha}{B_0}$ is finite. As $B_0\subseteq B$, we have that $A \harass B_0$.
\end{proof}

\begin{lemma}
\label{oppressBelowEliminates}
Let $C\subseteq\delta$. Assume there are some $\gamma<\delta$, $l_1<l_2<\CBr{\gamma}$ and $F_i\in\lrg{\gamma}{}{l_i}$ such that $C\oppress F_i$ and $\neighboursInSet{c}{F_2}\notin\lrg{\gamma}{}{l_2}$ for every $c\in C$. Then for every $\theta<\omega^2$, there is some independent, closed in its supremum, $X\subseteq F_1\cup F_2 \subseteq \delta$ with $\otp{X} = \theta$.
\end{lemma}

\begin{proof}
By Observation \ref{middleLayerLarge}, we may thin out $F_2$ so that $\subtree{=l_1}{\beta}\cap F_1\in \lrg{\beta}{}{l_1}$, for all $\beta\in F_2$.
By thinning out $C$ to a cofinal subset, we may assume $\otp{C} = \omega$. By Lemma \ref{oppressImpliesHarass} there is some $A_2\subcof F_2$ with $C\harass A_2$. By Lemma \ref{harassmentLemma}, we may assume that $\neighboursInSet{\beta}{C}$ is cofinite in $C$ for each $\beta\in A_2$. Without loss of generality, assume $\otp{A_2}=\omega$ and enumerate $A_2 = \setcolon{\beta_i}{i<\omega}$.

Let $n\in\omega$ be arbitrary, we find in $\delta$ an independent copy of $\omega\cdot (n+1)$. The set $C\cap \bigcap_{i=0}^{2n} \neighbours{\beta_i}$ is cofinite in $C$, so without loss of generality equal to $C$. For each $c\in C$, let $I_c =\setarg{i<2n+1}{\neighbours{c}\cap \subtree{=l_1}{\beta_i}\in \lrg{\beta_i}{}{l_1}}$. By the pigeonhole principle we may assume $I:=I_c$ is constant for all $c\in C$.

If $|I|> n$, choose some $\alpha\in C$, and for each $i\in I$ some $X_i\subcof \neighbours{\alpha}\cap \subtree{=l_1}{\beta_i}\cap F_1$ with $\otp{X_i} = \omega$. Conclude that $\bigcup_{i\in I} (X_i\cup\set{\beta_i})$ contains an independent copy of $\omega\cdot (n+1)$.

Then assume $|I| \leq n$, and without loss of generality $\set{0,\dots,n}\cap I = \emptyset$. For each $i\leq n$, let $B^i_0\subcof \subtree{=l_1}{\beta_i}\cap F_1$ be as guaranteed by Lemma \ref{oppressImpliesHarass}. As $C\oppress \subtree{=l_1}{\beta_j}\cap F_1$ for all $j<\omega$, for $i\notin I$ we have $C\harass B^i_0$. Iterating Lemma \ref{harassmentLemma}, construct a descending chain $C \supseteq C_0 \supseteq C_1 \dots \supseteq C_n$ of infinite sets such that $\neighboursInSet{\alpha}{C_i}$ is cofinite in $C_i$ for cofinitely many $\alpha\in B^i_0$. Without loss of generality, for every $\alpha\in\bigcup_{i\leq n} B^i_0$ the set $\neighboursInSet{\alpha}{C_n}$ is cofinite in $C_n$. Then for any $\alpha_1,\alpha_2\in \bigcup_{i=0}^n (B^i_0\cup\set{\beta_i})$, the sets $\neighbours{\alpha_1}$,$\neighbours{\alpha_2}$ intersect in a cofinite subset of $C_n$. So $\bigcup_{i=0}^n (B^i_0\cup\set{\beta_i})$ is an independent set containing a copy of $\omega\cdot(n+1)$.

Since $n$ was chosen arbitrarily, we find an independent copy of any $\theta <\omega^2$.
\end{proof}

\section{Proof of the theorem}
\label{Theorem}

For this section, it is helpful to use Appendix \ref{figure}. Fix $\delta = \omega^3\cdot 2$ and write $\layer{i}{j} = \setarg{\alpha\in\cc{\delta}{i}}{\CBr{\alpha} = j}$.

\begin{prop}
\label{upperBound}
$\closedramseynumber{\omega\cdot 2}{3} \leq \omega^3\cdot 2$
\end{prop}

\begin{proof}
Assume that the statement is false. Let $\col:\pairs{\delta}\to 2$ contradict the statement and let $G = (\delta, E)$ be the triangle-free graph corresponding to $\col$. By Proposition \ref{canonicalSkeleton} we may assume that $\col$ is canonical.

By item $1$ of Lemma \ref{edgeColourScarce}, there are $j_1 < j_2 < 3$ such that $\omega^3\notin \neighbours{\layer{1}{j_1}\cup\layer{1}{j_2}}$. By item $2$, there are $l_1 < l_2 < 3$ such that $\domsep{1}{3}{2}{l_1} = \domsep{1}{3}{2}{l_2} = 0$, and by thinning out to a skeleton we may assume $\omega^3\notin\neighbours{\layer{2}{l_1}\cup\layer{2}{l_2}}$.

Assume for a moment that $\layer{1}{j_t}\oppress\layer{2}{l_s}$ for every $t,s\in\set{1,2}$. By Lemma \ref{oppressBelowEliminates} it is impossible that $\domsep{1}{j_t}{2}{l_2} = 0$. So $\domsep{1}{j_1}{2}{l_2} = \domsep{1}{j_2}{2}{l_2} = 1$ in contradiction to item $3$ of Lemma $\ref{edgeColourScarce}$.

Let, then, $t,s\in\set{1,2}$ be such that $\layer{1}{j_t}\mathrel{\not\oppress}\layer{2}{l_s}$. Let $X_1\subcof\layer{1}{j_t}$ be such that $X_2:=\layer{2}{l_s}\setminus\neighbours{X_1}$ is infinite. By Ramsey's theorem and triangle-freeness, each $X_i$ contains an infinite independent set. We conclude that $X_1\cup\set{\omega^3}\cup X_2$ contains an independent copy of $\omega\cdot 2$ closed in its supremum, contradicting our assumption.
\end{proof}

\begin{prop}
\label{lowerBound}
	$\closedramseynumber{\omega\cdot 2}{3} \geq \omega^3\cdot 2$
\end{prop}

\begin{proof}
	We define a (canonical) triangle-free graph $G = (\delta,E)$ such that the induced subgraph on any $\theta<\delta$ does not contain an independent closed copy of $\omega\cdot 2$. Appendix \ref{colouringFigure} is a schematic representation of $G$. The dashed and dotted lines in the figure imply the $\oppress$ relation between the circled set more to the left-hand side and the circled set (or sets) more to the right-hand side.
	
	We use the notation 
	\[
	\edges{i,j}{k,l}:=\setarg{(\beta,\alpha)\in E}{\beta<\alpha, \beta\in\layer{i}{j}, \alpha\in\layer{k}{l}}.
	\]
	We define $E$ so that $\descol{i}{j+1}{j} = 1$ for all $i,j$.
	In our notation,
	\[
	\edges{i,j}{i,j+1} \supseteq \setarg{(\beta,\alpha)\in \layer{i}{j}\times\layer{i}{j+1}}{\beta\treedescendant\alpha}
	\]
	We list the sets of edges that imply $\domsep{i_1}{j_1}{i_2}{j_2} = 1$:
	\begin{align*}
	\edges{1,3}{2,0} &= \set{\omega^3}\times \layer{2}{0}
	\\
	\edges{1,2}{1,0} &= \setarg{(\omega^2\cdot k,~\omega^2\cdot k' +\omega\cdot l' +m')}{0<k\leq k'}
	\\
	\edges{1,0}{1,2} &= \setarg{(\omega^2\cdot k + \omega\cdot l + m,~\omega^2\cdot k')}{(k+1) + l < k'}
	\\
	\edges{1,0}{2,1} &= \setarg{(\omega^2\cdot k + \omega\cdot l + m,~ \omega^3 + \omega^2\cdot k' +\omega\cdot l')}{k+l < l'}
	\\
	\edges{2,2}{2,0} &=  \setarg{(\omega^3 + \omega^2\cdot k,~\omega^3 + \omega^2\cdot k' +\omega\cdot l' +m')}{k\leq k'}
	\end{align*}
	Lastly, we list sets of edges that imply $\oppress$ while $\domsep{i_1}{j_1}{i_2}{j_2} = 0$:
	\begin{align*}
	\edges{1,1}{2,1} = &\setarg{(\omega^2\cdot k + \omega\cdot l,~ \omega^3 +\omega^2\cdot k' + \omega\cdot l')}{l'<k}
	\\
	\edges{1,0}{1,1} \supseteq &\setarg{(\omega^2\cdot k + \omega\cdot l + m,~\omega^2\cdot k' + \omega\cdot l')}{0< k'- k < l, l'<l}
	\\
	\edges{1,0}{2,0} = &\setarg{(\omega^2\cdot k + \omega\cdot l + m,~ \omega^3 + \omega^2\cdot k' + \omega\cdot l' + m')}{l'<l}
	\\
	\edges{2,0}{2,1} \supseteq &\setarg{(\omega^3 + \omega^2\cdot k + \omega\cdot l + m,~ \omega^3 + \omega^2\cdot k' + \omega\cdot l')}{k < k', l' < l}
	\end{align*}
	
	In the last list, for $\edges{1,0}{1,1}$ and $\edges{2,0}{2,1}$ we do not use the equality sign in order to accommodate edges whose existence was declared previously. We let $E$ be minimal under inclusion and fulfilling our stated requirements.
	
	\medskip
	\noindent\textbf{Claim 1.} $G = (\delta,E)$ is triangle-free
	\begin{proof}
		Note that $\layer{i}{j}$ is independent for all $i,j$.
		
		Assume $\set{\omega^3,\alpha,\beta}$ is a triangle with $\alpha<\beta$. Then it must be that $\alpha\in \layer{1}{2}$ and $\beta\in\layer{2}{0}$. But then, it cannot be that $\set{\alpha,\beta}\in E$. So $\omega^3$ takes no part in a triangle.
		
		Consider $\gamma\in\layer{1}{2}$. Assume $T = \set{\alpha,\beta,\gamma}$ is a triangle with $\alpha<\beta$. Looking at $\neighbours{\layer{1}{2}}$, it must be that $T\subseteq\cc{\delta}{1}$. Since $\omega^3\notin T$ it must be that $\set{\CBr{\alpha},\CBr{\beta}} = \set{0,1}$. Since $\edges{1,1}{1,0} = \emptyset$, we have $\CBr{\alpha} = 0$ and $\CBr{\beta} = 1$. Since $\edges{1,2}{1,1} = \emptyset$ we have that $(\beta,\gamma)\in \edges{1,1}{1,2}$, implying $\beta\treeorder\gamma$ and $\alpha<\beta<\gamma$. From $(\alpha,\gamma)\in \edges{1,0}{1,2}$ we get $\alpha \not\treeorder\gamma$ and so $\alpha\not\treeorder\beta$. So
		\begin{align*}
		\alpha &= \omega^2\cdot k + \omega\cdot l + m
		\\
		\beta &= \omega^2\cdot (k'-1) + \omega\cdot l'
		\\
		\gamma &= \omega^2\cdot k'
		\end{align*}
		with the restrictions $(k'-1) - k < l$ and $(k+1) + l < k'$ coming from $(\alpha,\beta)\in \edges{1,0}{1,1}$ with $\alpha\not\treeorder\beta$ and $(\alpha,\gamma)\in\edges{1,0}{1,2}$ respectively. These two inequalities are contradictory, so $\gamma$ cannot take part in a triangle.
		
		Consider $\beta\in\layer{1}{1}$. Assume $\set{\alpha,\beta,\gamma}$ is a triangle with $\alpha<\gamma$. We already have enough to guarantee that there are no triangles within $\cc{\delta}{1}$. So $\gamma\in\cc{\delta}{2}$ and in particular $(\beta,\gamma)\in\edges{1,1}{2,1}$. This leaves no other option but $(\alpha,\beta)\in \edges{1,0}{1,1}$.
		So
		\begin{align*}
		\alpha &= \omega^2\cdot k + \omega\cdot l + m
		\\
		\beta &= \omega^2\cdot k' + \omega\cdot l'
		\\
		\gamma &= \omega^3 + \omega^2\cdot k'' + \omega\cdot l''
		\end{align*}
		with the restrictions $k+l<l''$ and $l''<k'$ coming from $(\alpha,\gamma)\in\edges{1,0}{2,1}$ and $(\beta,\gamma)\in\edges{1,1}{2,1}$ respectively, which gives $k+l<k'$. As $k'\neq k+1$, it cannot be that $\alpha\treedescendant\beta$. Thus, we have the additional restriction $k'-k<l$ from $(\alpha,\beta)\in\edges{1,0}{1,1}$. This again is contradictory, and so $\beta$ takes no part in a triangle.
		
		Consider $\beta\in\layer{2}{2}$. Assume $\set{\alpha,\beta,\gamma}$ is a triangle with $\alpha<\gamma$. It must be that $\alpha<\beta<\gamma$, $\alpha\treedescendant\beta$ and $\gamma\in\layer{2}{0}$. But $\edges{2,1}{2,0} = \emptyset$, and so $(\alpha,\gamma)\in E$ cannot happen, contrary to our assumption.
		
		Finally, assume $\set{\alpha,\beta,\gamma}$ is a triangle with $\alpha<\beta<\gamma$. The only possibility left is $\alpha\in \layer{1}{0}$, $\beta\in\layer{2}{0}$, $\gamma\in\layer{2}{1}$. So
		\begin{align*}
		\alpha &= \omega^2\cdot k + \omega\cdot l + m
		\\
		\beta &= \omega^3 + \omega^2\cdot k' + \omega\cdot l' + m'
		\\
		\gamma &= \omega^3 + \omega^2\cdot k'' + \omega\cdot l''
		\end{align*}
		with $l'<l$ coming from $(\alpha,\beta)\in\edges{1,0}{2,0}$ and $k+l <l''$ coming from $(\alpha,\gamma)\in\edges{1,0}{2,1}$. If $\beta\treedescendant\gamma$, then $l'' = l'+1$, leading to $k+l<l'+1<l+1$ which is impossible. So this is not the case, and we get the extra restriction $l''<l'$ from $(\beta,\gamma)\in\edges{2,0}{2,1}$. This gives us $k+l<l'<l$ in contradiction. All possibilities have been exhausted, and so there are no triangles in $G$.
	\end{proof}
	
	\medskip
	\noindent\textbf{Claim 2.} Every independent copy of $\omega\cdot 2$ contained in $G$, which is closed in its supremum, is cofinal in $\delta$.
	\begin{proof}
	Let $W:= A\cup\set{\tau}\cup B\subseteq \delta$ be an independent set with $A<\tau<B$, $\otp{A} = \otp{B} = \omega$ and $\tau=\sup A$. Assume for the purpose of a contradiction that $W\subseteq \theta$ for some $\theta<\delta$. Without loss of generality, let $\theta = \omega^3 + \omega^2\cdot n$ for some natural $n$. For $i,j$  denote $\layerTheta{\theta}{i}{j} = \setarg{\alpha\in\cc{\theta}{i}}{\CBr{\alpha}=j}$.
	
	 By the pigeonhole principle, we may assume that $A\subcof\subtree{=j_A}{\tau}$ for some $j_A<\CBr{\tau}$, and $B\subseteq \layerTheta{\theta}{i_B}{l_B}$ for some $i_B\leq k_\theta$, $l_B<\CBr{\sup\cc{\theta}{i_B}}$. Fix such $j_A,i_B,l_B$. We proceed by eliminating where $\tau$ can potentially reside.
	
	As $\tau$ is a limit ordinal, $\CBr{\tau} > 0$. Since $\subfan{\tau}\times \set{\tau}\subseteq E$, it also must be that $\CBr{\tau} > 1$ and $j_A <\CBr{\tau}-1$.
	
	Assume $\tau\in\layerTheta{\theta}{1}{3}$, that is, $\tau = \omega^3$. Then clearly $i_B>1$, $l_B = 1$. For every $\alpha\in\layerTheta{\theta}{1}{0}$ the set $\layerTheta{\theta}{i_B}{1}\setminus\neighbours{\alpha}$ is finite, so $j_A\neq 0$. Also, $\layerTheta{\theta}{1}{1}\harass \layerTheta{\theta}{i_B}{1}$, and so $j_A\neq 1$. This contradicts $j_A<\CBr{t}-1 = 2$, and so $\tau\notin\layerTheta{\theta}{1}{3}$. Thus, it must be that $\CBr{\tau} = 2$ and $j_A = 0$.
	
	Assume $\tau\in\layer{1}{2}$. Note the following:
	\begin{itemize}
	\item 
	$\set{\alpha>\tau}\cap\layerTheta{\theta}{1}{0}\subseteq \neighbours{\tau}$;
	\item
	$\set{\alpha>\tau}\cap \layerTheta{\theta}{1}{1}\subseteq\neighbours{A}$;
	\item
	$\layerTheta{\theta}{1}{2}\setminus \neighbours{\alpha}$ is finite for every $\alpha\in A$.
	\end{itemize}
	The three listed items guarantee that $i_B>1$. However, for each $i>1$ note that
	\begin{itemize}
	\item
	$\layerTheta{\theta}{i}{0}\subseteq \neighbours{A}$;
	\item
	$\layerTheta{\theta}{i}{1}\setminus\neighbours{\alpha}$ is finite for every $\alpha\in A$.
	\end{itemize}
	So it cannot be that $l_B\in\set{0,1}$, a contradiction.
	
	So the only option is that $\tau\in\layerTheta{\theta}{i}{2}$ for some $1<i\leq k_\theta$. We now observe that for each $i'>i$
	\begin{itemize}
	\item
	$\layerTheta{\theta}{i'}{0}\subseteq \neighbours{\tau}$;
	\item
	$\layerTheta{\theta}{i'}{1}\subseteq \neighbours{A}$.
	\end{itemize}
	prohibiting $l_B = 0$ or $l_B=1$, and therefore the existence of $l_B$. We conclude that there is no $\theta$ as we had assumed.
	\end{proof}
	
	By the claims, taking the induced colouring on any $\theta < \omega^3\cdot 2$ demonstrates that $\theta\not\to_{cl}(\omega\cdot 2, 3)^{2}$.
\end{proof}

\begin{corollary}
$\closedramseynumber{\omega\cdot 2}{3} = \omega^3\cdot 2$ \qed
\end{corollary}
\appendix
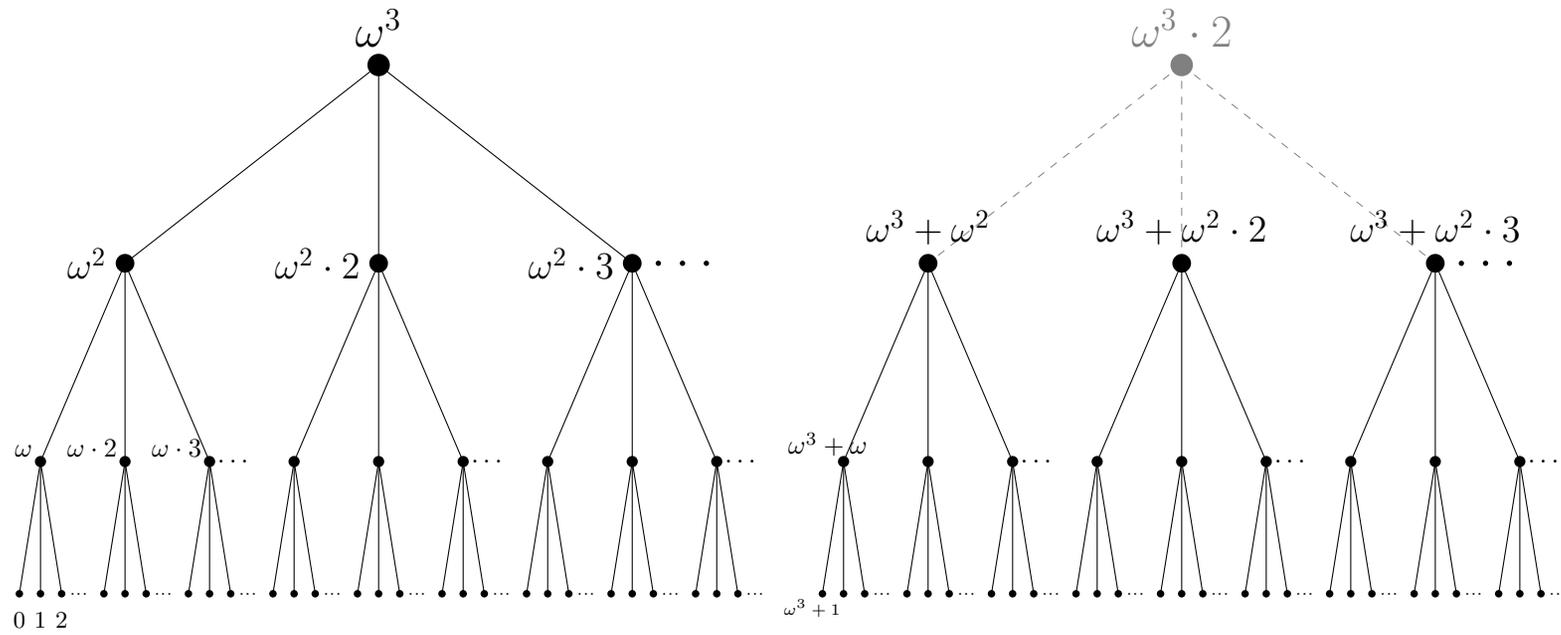
\begin{landscape}
\section{$\omega^3\cdot 2$ under the $\treeorder$ relation}
\label{figure}
\begin{figure}[b]
\begin{tikzpicture}[x=8,y=50]
\foreach \i in{0}
{
	\node[fill=black,circle, inner sep=3pt] at (\i,4) {};
	\foreach \j in{-12,0,12}
	{
		\node[fill=black,circle, inner sep=2.5pt] at (\i+\j,2.5) {};
		\draw (\i,4) -- (\i+\j,2.5);
		\foreach \k in{-4,0,4}
			{
			\node[fill=black,circle, inner sep=1.5pt] at (\i+\j +\k,1) {};
			\draw (\i+\j,2.5) -- (\i+\j +\k,1);
			\foreach \l in{-1,0,1}
				{
				\node[fill=black,circle, inner sep=1pt] at (\i+\j+\k+\l,0) {};
				\draw (\i+\j+\k,1) -- (\i+\j +\k+\l,0);
				}
			\node[label=right:\tiny $...$] at (\i+\j+\k+1/2,0) {};
			}
		\node[label=right:$\dots$] at (\i+\j+3.5,1) {};
	}
	\node[label=right:\Huge$\dots$] at (\i+12,2.5) {};
}

\foreach \i in{38}
{
	\node[fill=gray,circle, inner sep=3pt] at (\i,4) {};
	\foreach \j in{-12,0,12}
	{
		\draw[help lines,dashed] (\i,4) -- (\i+\j,2.5);
		\node[fill=black,circle, inner sep=2.5pt] at (\i+\j,2.5) {};
		\foreach \k in{-4,0,4}
			{
			\node[fill=black,circle, inner sep=1.5pt] at (\i+\j +\k,1) {};
			\draw (\i+\j,2.5) -- (\i+\j +\k,1);
			\foreach \l in{-1,0,1}
				{
				\node[fill=black,circle, inner sep=1pt] at (\i+\j+\k+\l,0) {};
				\draw (\i+\j+\k,1) -- (\i+\j +\k+\l,0);
				}
				\node[label=right:\tiny $...$] at (\i+\j+\k+1/2,0) {};
			}
		\node[label=right:$\dots$] at (\i+\j+3.5,1) {};
	}
	\node[label=right:\Huge$\dots$] at (\i+12,2.5) {};

}
\node[label=above:\huge$\omega^3$] at (0,4) {};
\node[label=left:\LARGE$\omega^2$] at (-12,2.5) {};
\node[label=left:\LARGE$\omega^2\cdot 2$] at (0,2.5) {};
\node[label=left:\LARGE$\omega^2\cdot 3$] at (12,2.5) {};
\node[label=north west:\normalsize$\omega$] at (-12-4+0.5,0.9) {};
\node[label=north west:\normalsize$\omega\cdot 2$] at (-12+0.5,0.9) {};
\node[label=north west:\normalsize$\omega\cdot 3$] at (-12+4+0.5,0.9) {};
\node[label=below:\small$0$] at (-12-4-1,0) {};
\node[label=below:\small$1$] at (-12-4,0) {};
\node[label=below:\small$2$] at (-12-4+1,0) {};

\node[label=above:\huge$\color{gray}\omega^3\cdot 2$] at (38,4) {};
\node[label=above:\LARGE$\omega^3 + \omega^2$] at (38-12,2.5) {};
\node[label=above:\LARGE$\omega^3 + \omega^2\cdot 2$] at (38+0,2.5) {};
\node[label=above:\LARGE$\omega^3 + \omega^2\cdot 3$] at (38+12,2.5) {};
\node[label=north west:\normalsize$\omega^3+\omega$] at (38-12-2,0.9) {};
\node[label=below:\tiny$\omega^3 + 1$] at (38-12-4-1.5,0.1) {};
\end{tikzpicture}
\end{figure}
\end{landscape}

\begin{landscape}
\section{Visualization of Lemma \ref{lowerBound}}
\label{colouringFigure}
\begin{figure}[b]
\scalebox{0.95}{
\begin{tikzpicture}[x=8,y=50]
\foreach \i in{0}
{
	\node[fill=black,circle, inner sep=3pt] at (\i,4) {};
	\foreach \j in{-12,0,12}
	{
		\node[fill=black,circle, inner sep=2.5pt] at (\i+\j,2.5) {};
		\draw[very thin] (\i,4) -- (\i+\j,2.5);
		\foreach \k in{-4,0,4}
			{
			\node[fill=black,circle, inner sep=1.5pt] at (\i+\j +\k,1) {};
			\draw[very thin] (\i+\j,2.5) -- (\i+\j +\k,1);
			\foreach \l in{-1,0,1}
				{
				\node[fill=black,circle, inner sep=1pt] at (\i+\j+\k+\l,0) {};
				\draw[very thin] (\i+\j+\k,1) -- (\i+\j +\k+\l,0);
				}
			}
	}
}

\foreach \i in{38}
{
	\node[fill=gray,circle, inner sep=3pt] at (\i,4) {};
	\foreach \j in{-12,0,12}
	{
		\draw[help lines,dashed] (\i,4) -- (\i+\j,2.5);
		\node[fill=black,circle, inner sep=2.5pt] at (\i+\j,2.5) {};
		\foreach \k in{-4,0,4}
			{
			\node[fill=black,circle, inner sep=1.5pt] at (\i+\j +\k,1) {};
			\draw (\i+\j,2.5) -- (\i+\j +\k,1);
			\foreach \l in{-1,0,1}
				{
				\node[fill=black,circle, inner sep=1pt] at (\i+\j+\k+\l,0) {};
				\draw (\i+\j+\k,1) -- (\i+\j +\k+\l,0);
				}
			}
	}

}
\draw[very thick,rounded corners] (0,4) .. controls (45,6) and (65,3) ..
   (55.5,0)
   (55.5,0.1) rectangle (20.5,-0.1);
\draw[very thick,rounded corners] (-12,2.5) .. controls (-7.5,1.6) and (-6.5,1.25) ..
	(-5,0.1)
	(-5.5,-0.1) rectangle (17.5,0.1);
\draw[very thick,rounded corners] (-7,0) .. controls (-6,1.5) and (-2,2) ..
	(-0.5,2.4)
	(-1,2.4) rectangle (13,2.6);
\draw[very thick,rounded corners] (17,0) --
	(21.5,1)
	(21.5,0.9) rectangle (54.5,1.1);
\draw[very thick,rounded corners] (26,2.5) parabola
	(33,0.2)
	(32.5,0.2) rectangle (56,-0.2);
\draw[color=green!40!black,very thick,densely dotted,rounded corners] (-18,0.7) rectangle (18,1.3)
	(18,1) -- (21,1)
	(21,0.8) rectangle (55,1.2);
\draw[color=red!60!black,very thick,densely dashed,rounded corners] (56.5,0.25) rectangle (20.25,-0.25)
	.. controls (-5,-0.5) ..
	(-6.25, -0.1)
	(-17.75,-0.1) rectangle (-6.25,0.1) --
	(-2,0.9)
	(-4.5,0.9) rectangle (16.5,1.1);
\draw[color=blue!60!black,ultra thick,dashed] (19.75,-0.35) rectangle (31.5,0.35)
	(30,0.35) -- (33,0.7)
	(32,0.7) rectangle (55.5,1.3);
	
\end{tikzpicture}
}
\end{figure}
\end{landscape}

\subsection*{Acknowledgements}
This research stemmed from conversations of the author with Jacob Hilton at the 8th "Young Set Theory" workshop in Jerusalem, October 2015. I thank Hilton for introducing me to the subject matter of this paper, and choosing the right riddles to get me hooked. I also thank the workshop's organizers and the Israeli Institute of Advanced Studies (IIAS) for providing such a stimulating environment.

I would also like to thank Bill Chen, Jacob Hilton, Menachem Kojman, Nadav Meir and Thilo Weinert for providing valuable feedback on early versions of this paper.

The research was partially supported by ISF grant No. 181/16 and 1365/14.

\bibliographystyle{alpha}
\bibliography{myrefs}
\end{document}